\documentclass[12pt,centertags]{amsart}
\usepackage{amsmath,amstext,amsthm,a4,amssymb,amscd}
\usepackage[mathscr]{eucal}
\usepackage{mathrsfs}
\usepackage{epsf}

\usepackage{amsmath,amstext,amsthm,a4,amssymb,amscd,mathrsfs}
\usepackage{mathtools}
\usepackage[mathscr]{eucal}
\usepackage{mathrsfs}
\usepackage{epsf}
\numberwithin{equation}{section}

\def\mE{\mathcal{E}}

\newtheorem{thm}{Theorem}[section]
\newtheorem{lemma}[thm]{Lemma}
\newtheorem{prop}[thm]{Proposition}

\theoremstyle{definition}
\newtheorem{rem}[thm]{Remark}
\theoremstyle{definition}
\newtheorem{ex}[thm]{Example}
\theoremstyle{definition}

\newcommand{\be}{\begin{eqnarray}}
\newcommand{\ee}{\end{eqnarray}}

\newcommand{\comment}[1]{}

\begin{document}

\title[Llarull's theorem on odd dimensional manifolds: the noncompact case]{Llarull's theorem on odd dimensional manifolds: \\ the noncompact case}
 
\author{Yihan Li, Guangxiang Su, Xiangsheng Wang and Weiping Zhang}

\address{School of Mathematical Sciences \& LPMC, Nankai
University, Tianjin 300071, P.R. China}
\email{yhli@nankai.edu.cn}

\address{Chern Institute of Mathematics \& LPMC, Nankai
University, Tianjin 300071, P.R. China}
\email{guangxiangsu@nankai.edu.cn}

\address{School of Mathematics, Shandong University, Jinan, Shandong 250100, P.R. China}
\email{xiangsheng@sdu.edu.cn}

\address{Chern Institute of Mathematics \& LPMC, Nankai
University, Tianjin 300071, P.R. China}
\email{weiping@nankai.edu.cn}

\begin{abstract}
Let $(M,g^{TM})$ be an odd dimensional ($\dim M\geq 3$) connected oriented noncompact complete spin Riemannian manifold. Let $k^{TM}$ be the associated scalar curvature. Let $f:M\to S^{\dim M}(1)$ be a smooth area decreasing map which is locally constant near infinity and of nonzero degree. If $k^{TM}\geq ({\dim M})({\dim M}-1)$ on the support of ${\rm d}f$, then we show that $\inf(k^{TM})<0$. This answers a question of Gromov. 
\end{abstract}

\maketitle



\section{Introduction} \label{s0}

It is well-known that, starting with the Lichnerowicz vanishing theorem \cite{L63}, Dirac operators have played important roles in the study of Riemannian metrics of positive scalar curvature on spin manifolds  (cf. \cite{GL83}, \cite{LaMi89}). A notable example is Llarull's rigidity theorem \cite{Ll} which states that for a compact spin Riemannian manifold $(M,g^{TM})$ of dimension $n$ such that the associated scalar curvature 
$k^{TM}$ verifies that $k^{TM}\geq n(n-1)$, then any (non-strictly) area decreasing smooth map $f:M\to S^n(1)$ of nonzero degree is an isometry, where $S^n(1)$ is the standard unit $n$-sphere.

In answering a question of Gromov in an earlier version of \cite{G23}, Zhang \cite{Z20} proved that for an even dimensional noncompact complete spin Riemannian manifold $(M,g^{TM})$ and a smooth (non-strictly) area decreasing map $f:M\to S^{\dim M}(1)$ which is locally constant near infinity and of nonzero degree, if the associated scalar curvature $k^{TM}$ verifies
\begin{equation}\label{m1.1}
  k^{TM}\geq (\dim M) (\dim M-1)\text{ on }{\rm Supp}({\rm d}f),
\end{equation}
then $\inf (k^{TM})<0$. When $\dim M$ is odd, Zhang \cite{Z20} proved that $\inf (k^{TM})<0$ still holds if the inequality in (\ref{m1.1}) is strict, by using the standard trick of passing $M$ to $M\times S^1$.  The purpose of this paper is to improve Zhang's result in the odd dimensional case so that one gets a complete  answer to  Gromov's question.

Recall that the main idea in \cite{Z20}, which goes back to \cite[(1.11)]{Z19}, 
is to deform the involved twisted Dirac operator on $M$ by a suitable endomorphism of a twisted  ${\mathbb Z}_2$-graded vector bundle. {Since the deformed Dirac operator is invertible near infinity, one can apply} a relative index theorem (cf. \cite[Theorem 2.1]{Z23}) to {obtain a contradiction}.

In \cite{LSW}, using the spectral flow method as suggested in \cite{G23}, Li, Su and Wang gave a direct proof of Llarull's rigidity theorem for odd dimensional manifolds. 

In this paper, we combine the methods in \cite{Z20} and \cite{LSW} to deal with the odd dimensional noncompact case. In doing so, we also use the ideas in \cite{GL83} and \cite{SWZ} to construct a closed manifold by a gluing method, to overcome the obvious difficulty caused by the noncompactness of the underlying manifold.

The main result of this paper can be stated as follows.
\begin{thm}\label{th1.1}
Let $(M,g^{TM})$ be an odd dimensional ($\dim M\geq 3$) connected oriented noncompact complete spin Riemannian manifold. Let $k^{TM}$ be the associated scalar curvature. Let $f:M\to S^{\dim M}(1)$ be a smooth area decreasing map which is locally constant near infinity and of nonzero degree. Suppose 
\begin{align}
k^{TM}\geq ({\dim M})(\dim M-1)\ \ {\rm on}\ \ {\rm Supp}({\rm d}f),
\end{align}
then $\inf(k^{TM})<0$.
\end{thm}

\begin{ex}
 By  Cecchini-Zeidler (\cite[Proposition 5.2]{CZ}), which goes back to Gromov-Lawson (\cite{GL83}),  for any torus   ${\rm T}^{n-1}$   of dimension $n-1$,    there exists a Riemannian metric $g^{T({\rm T}^{n-1}\times [-1,1])}$ on ${\rm T}^{n-1}\times [-1,1]$ such that the scalar curvature satisfies $k^{T({\rm T}^{n-1}\times [-1,1])}=n(n-1)$. Moreover, there exists a smooth area decreasing map 
$f:{\rm T}^{n-1}\times [-1,1]\to S^n(1)$ such that $({\rm Supp}({\rm d}f))\cap ( \partial({{\rm T}^{n-1}\times [-1,1]}))=\emptyset$ and ${\rm deg}(f)\neq 0$. One can extend two ends of ${\rm T}^{n-1}\times [-1,1]$ to get a  complete Riemannian manifold ${\rm T}^{n-1}\times (-\infty,+\infty)$ with cylindrical ends. Then Theorem \ref{th1.1}, together with \cite[Theorem 2.1]{Z20}, claims that  for any such extension, there exists $x\in {\rm T}^{n-1}\times (-\infty,+\infty)$ such that $k^{T({\rm T}^{n-1}\times (-\infty,+\infty))}(x)<0$ for all $n\geq 2$.
\end{ex}

\begin{rem}
Gromov (cf. \cite[p. 188]{G23}) has proposed another noncompact extension of the Llarull theorem, which states that for
 any  proper length non-increasing map $f:(M,g^{TM})\to S^n(1)\setminus \Sigma $ of non-zero degree from any Riemannian spin manifold $M$ (with  scalar curvature $k^{TM}\geq n(n-1)$) to  $S^n(1)$ with a finite  set $\Sigma$ of points removed,  $f$ must be an isometry. Our example above   shows again that (at least in  the single point punctured case),    the area decreasing and the distance decreasing properties are significantly different in the noncompact setting.

\end{rem}

\begin{rem}
By combining the methods in \cite{LSW} and this paper with those in \cite{C}, one can improve \cite[Theorem A]{C}, in the odd dimensional   case, where  the condition “$f$ is constant in a neighborhood of $\partial X$" can be relaxed to “$f$ is locally constant in a neighborhood of $\partial X$" (cf. \cite{Shi}).
\end{rem}

We briefly explain the idea of the proof of Theorem~\ref{th1.1}.
We argue by contradiction, i.e.,\ we assume that $\inf (k^{TM})\geq 0$.

Following~\cite{GL83}, we choose a compact hypersurface $\Sigma$ in $M$ that cuts off a compact part $N$.
We then form a closed manifold $\widetilde{M}$ by attaching a product neck and an oppositely oriented copy of $N$ to this compact part $N$.
The given map $f$ is extended to this closed manifold $\widetilde{M}$ in such a way that the degree is preserved.

On this closed manifold $\widetilde{M}$, we combine the ideas in~\cite{LSW} and~\cite{Z20} to construct a family of operators depending on a parameter $u\in [0,1]$.
Each element of this family is a twisted Dirac operator deformed by a zero-order term $V_{\varepsilon, \beta}$, where $\varepsilon, \beta$ are two additional parameters.
If we choose the metric on $\widetilde{M}$ and the parameters $\varepsilon, \beta$ suitably, then $\inf (k^{TM})\geq 0$ implies that this family of deformed twisted Dirac operators is invertible for all $u\in [0,1]$.
Moreover, by the construction of this family of operators, the two endpoint operators in the family are conjugate.
Combining these two facts, a calculation of the $\eta$-invariant yields that the degree of $f$ vanishes, which contradicts the condition on $f$.

The rest of this paper is organized as follows. In Section 2, we recall the definitions of the Clifford algebras and the supertrace. In Section 3, we construct a closed manifold by a gluing method and define a family of deformed twisted Dirac operators. In Section 4, we prove the invertibility of this family of deformed twisted Dirac operators. In Section 5, we prove Theorem \ref{th1.1}  by using the $\eta$-invariant   of Atiyah-Patodi-Singer \cite{APS}.

\section{Clifford algebras and the supertrace}

In this section we briefly recall some properties of the Clifford algebras (cf. \cite{BC}, \cite{BF}, \cite{LaMi89}).

Let $\mathbb{R}^{2k}$ be the $2k$ dimensional real oriented Euclidean space with the standard metric. Let $\{\partial_1,\cdots,\partial_{2k}\}$ be an oriented orthonormal basis of $\mathbb{R}^{2k}$. Let $Cl(\mathbb{R}^{2k})$ be the Clifford algebra of $\mathbb{R}^{2k}$ generated with the relations
\begin{align}\label{c2.1}
\partial_i\partial_j+\partial_j\partial_i=-2\delta_{ij},\ \ 1\leq i,j\leq 2k.
\end{align}
We have,
$${\mathbb{C}l}(\mathbb{R}^{2k}):=Cl(\mathbb{R}^{2k})\otimes _{\mathbb{R}}\mathbb{C}\cong {\rm End}(S_{2k}),$$
where $S_{2k}=S_{2k,+}\oplus S_{2k,-}$ is the $\mathbb{Z}_2$-graded spinor space.
Let $\Gamma=(\sqrt{-1})^{k}\partial_1\cdots\partial_{2k}$ be the chirality operator such that
$\Gamma|_{S_{2k,\pm}}=\pm {\rm Id}|_{S_{2k,\pm}}$.  We equip $S_{2k}$ with the canonical Hermitian metric.

Similarly, let $\mathbb{R}^{2k-1}\subset \mathbb{R}^{2k}$ be generated by $\{\partial_1,\cdots,\partial_{2k-1}\}$ with the standard induced metric.  We also have the Clifford algebras $Cl(\mathbb{R}^{2k-1})$, $\mathbb{C}l(\mathbb{R}^{2k-1})$ and the spinor space $S_{2k-1}$ which is identified with $S_{2k,+}$ (cf. \cite{BF}). As in \cite{BC} and \cite{BF},
for any $v\in  \mathbb{R}^{2k-1}$, we denote by $ {c}(v)=v\partial _{2k}$ the action of $v$ on $S_{2k-1}=S_{2k,+}$.

Let $W=W_{+}\oplus W_{-}$ be a $\mathbb{Z}_2$-graded vector space and $\rho$ be the $\mathbb{Z}_2$-grading operator such that $\rho|_{W_{\pm}}={\pm {\rm Id}}|_{W_\pm}$. Then ${\rm End}(W)$ also has a $\mathbb{Z}_2$-grading ${\rm End}(W)={\rm End}_{+}(W)\oplus {\rm End}_{-}(W)$ defined by
$${\rm End}_{+}(W)={\rm End}(W_{+})\oplus {\rm End}(W_{-}),$$
$${\rm End}_{-}( W)={\rm Hom}(W_-,W_{+})\oplus {\rm Hom}(W_+,W_-).$$

Following \cite{Q},  for any $A\in{\rm End}(W)$, in terms of the $\mathbb{Z}_2$-grading, it can be written in the form
\begin{align}
A=\left(\begin{matrix}
A_{00}&A_{01}\\A_{10}&A_{11}
\end{matrix}\right).
\end{align}
Then the supertrace of $A$ is defined by 
\begin{align}
{\rm tr_s}(A)={\rm tr}(\rho A)={\rm tr}(A_{00})-{\rm tr}(A_{11}).
\end{align}

Following \cite[\S 1]{BC}, for any  $v\in \mathbb{R}^{2k-1}$, 
 we extend the action of $c(v)$ to  ${\rm End}(S_{2k-1}{\otimes} W) $ as  $c(v)\otimes \rho$. We still denote this extended action by $c(v)$.
Also, for any $A\in {\rm End}(W)$, we extend it to act on $S_{2k-1}\otimes W$ as ${\rm Id}_{S_{2k-1}}\otimes A$, and still denote the extended action by $A$. Then for any $A\in {\rm End}_-(W)$, one has 
 \begin{align}\label{2.00}
 c(v)A+Ac(v)=0. 
 \end{align}

For any multi-index $I=\{i_1,\cdots, i_q\}\subset \{1,2,\cdots,2k-1\}$,  denote $\partial_I=\partial_{i_1}\cdots \partial_{i_q}$, $ {c}(\partial_I)= {c}(\partial_{i_1})\cdots  {c}(\partial_{i_q})$ and $|I|=q$.

\begin{lemma}\label{le2.1}
Let $T\in {\rm End}(S_{2k-1}{\otimes}W)$ be an element of the form
\begin{align}\label{5.0}
T=\sum_{|I|\ {\rm odd}}c\left(\partial_I\right)  A_{I}+\sum_{|J|\ {\rm even}} c\left(\partial_J\right) B_{J},
\end{align}
where $A_I\in {\rm End}_{+}(W), B_J\in {\rm End}_{-}(W)$. Then we have
\begin{align}\label{5.0a}
{\rm tr}(T) 
 =(\sqrt{-1})^{-k}2^{k-1}{\rm tr_s}\left(\left. A_{I_{2k-1}}\right|_W\right),
\end{align}
where $I_{2k-1}=\{1,2,\cdots,2k-1\}$.
\end{lemma}
\begin{proof}
From (\ref{5.0}), we have
\begin{align}\label{5.1}
{\rm tr}(T)=\sum_{|I|\ {\rm odd}}{\rm tr}\left( {c}(\partial_I)|_{S_{2k-1}}\right){\rm tr}\left( \left. \left(\rho A_{I}\right)\right|_W\right)+\sum_{|J|\ {\rm even}}{\rm tr}\left( {c}(\partial_J)|_{S_{2k-1}}\right){\rm tr}\left(\left.B_{J}\right|_W\right).
\end{align}

Since $B_J\in {\rm End}_{-}(W)$, we have ${\rm tr}(B_J|_W)=0$. By \cite[Lemma 1.22]{BC} and \cite[(1.7)]{BF}, for $|I|\geq 1$, ${\rm tr}\left( {c}(\partial_I)|_{S_{2k-1}}\right)$ is nonzero iff $I=I_{2k-1}=\{1,2,\cdots,2k-1\}$ and 
\begin{align}\label{5.2}
{\rm tr}\left( {c}(\partial_I)|_{S_{2k-1}}\right)=(\sqrt{-1})^{-k}2^{k-1}.
\end{align}

From (\ref{5.1}) and (\ref{5.2}), we get (\ref{5.0a}). 

\end{proof}

On the other hand, for the role of $W$, 
we consider  the standard unit sphere $S^{2k-1}(1)$ as a subset of $\mathbb{R}^{2k}$ and for any $x\in S^{2k-1}(1)$ we write $x=(x^1,x^2,\cdots,x^{2k})$ with respect to the basis $\{\partial_1,\cdots,\partial_{2k}\}$. 

Let $E=E_+\oplus E_-=(S^{2k-1}(1)\times S_{2k,+})\oplus (S^{2k-1}(1)\times S_{2k,-})$ be the $\mathbb{Z}_2$-graded trivial Hermitian vector bundle of spinors on $S^{2k-1}(1)$. Let ${\rm d}={\rm d}_{+}\oplus {\rm d}_{-}$ be the trivial Hermitian connection on $E$. The chirality operator $\Gamma$ induces an action $\rho$ on $E$ such that $\rho|_{E_{\pm}}=\pm {\rm Id}|_{E\pm}$.

For any $u\in \mathbb{R}^{2k}$, denote the Clifford action of $u$ on $S_{2k}$ by ${\widetilde c}(u)$. 
Let $g:S^{2k-1}(1)\to {\rm End}_{+}(S_{2k})$ be defined by
\begin{align}
g(x)=\widetilde{c}(\partial_{2k})\widetilde{c}(x).
\end{align} 
Clearly, $g$ is unitary.

We define a family of connections $\{\nabla^{E,u},{0\leq u\leq 1}\}$ on $E$ by 
\begin{align}
\nabla^{E,u}={\rm d}+ug^{-1}[{\rm d},g]=:{\rm d}+u\omega,
\end{align}
where $[\cdot,\cdot]$ is the supercommutator in the sense of \cite{Q} and $\omega=g^{-1}[{\rm d},g]$ is an ${\rm End}_{+}(E)$-valued 1-form.  Under the decomposition $E=E_+\oplus E_-$, the connection $\nabla^{E,u}$ has the form
\begin{align}
\nabla^{E,u}=\left(\begin{matrix}\nabla^{E_+,u}&0\\0&\nabla^{E_-,u}\end{matrix}\right)=\left(\begin{matrix}{\rm d}_++u\omega&0\\0&{\rm d}_{-}+u\omega\end{matrix}\right).
\end{align}
By \cite[Lemma 2.2]{LSW}, for any $u\in [0,1]$, $\nabla^{E,u}$ is a unitary connection on $E$.

By \cite[Proposition 1.4]{Ge}, one gets
\begin{align}\label{d4.11}
{\rm tr_s}\left(\omega^{2k-1}\right)=(-2\sqrt{-1})^k (2k-1)!\, {\rm vol}_{S^{2k-1}(1)},
\end{align}
where ``${\rm vol}_{S^{2k-1}(1)}$" is the Riemannian volume form on $S^{2k-1}(1)$. Note that the sign here is different from that in \cite[Proposition 1.4]{Ge}, since we use the relations (\ref{c2.1})
which differ from what is used in \cite{Ge} by a negative sign.

\section{The glued manifold and the deformed twisted Dirac operators}

This section consists of two subsections. In the first subsection,    we reduce the ``locally constant near infinity" situation to the case of ``constant near infinity", and then construct a closed glued manifold by  using ideas in \cite{GL83} and \cite{SWZ}. In the second subsection we construct a family of deformed Dirac operators on the glued manifold.

Let $(M,g^{TM})$ be a $2k-1$ dimensional connected oriented noncompact complete spin Riemannian manifold, $k\geq 2$. Let $f:M\to S^{2k-1}(1)$ be a smooth area decreasing map which is locally constant near infinity and of nonzero degree. Here the area decreasing means that for any two-form $\alpha\in\Omega^{2}(S^{2k-1}(1))$, $f^*\alpha\in\Omega^2(M)$ verifies that
\begin{align}\label{6.0}
|f^*\alpha|\leq |\alpha|.
\end{align}

\subsection{Construction of a closed manifold by a gluing method}

Following \cite[Theorem 1.17]{GL83}, we choose a fixed point $x_0\in M$ and let $d:M\to \mathbb{R}^+$ be a regularization of the distance function ${\rm dist}(x,x_0)$ such that 
{
  \begin{equation}
    \label{eq:fun-d}
  |\nabla d|(x)\leq \frac{3}{2},
\end{equation}
for any $x\in M$.}
Set
\begin{equation}\label{728}
  B_m=\{x\in M: d(x)\leq m\},\ m\in \mathbb{N}.
\end{equation}

Let $K\subset M$ be a compact subset such that $f$ is locally constant outside $K$. Since $K$ is compact, we can choose a sufficiently large $m$ such that $K\subseteq B_m$. This implies
\begin{equation}
  \label{eq:df-b}
  {\rm Supp}({\rm d} f) \subseteq K \subseteq B_m.
\end{equation}

Following \cite{GL83}, we take a compact hypersurface $H_{3m}\subseteq M\setminus K$, cutting $M$ into two parts such that the compact part, denoted by $M_{H_{3m}}$, contains $B_{3m}$. Then $M_{H_{3m}}$ is a compact smooth manifold with boundary $H_{3m}$. Note that the number of connected components of $M\setminus M_{H_{3m}}$ is finite. Let $\{Y_j\}_{j=1}^{l}$ be the connected components of $M\setminus {M^{\circ}_{H_{3m}}}$. 

Let $H_{3m}\times [-1,2]$ be the product manifold and we construct a metric on $H_{3m}\times [-1,2]$ as follows. Near the boundary $H_{3m}\times \{-1\}$ of $H_{3m} \times [-1, 2]$, i.e.,\ $H_{3m}\times [-1, -1+ \varepsilon']$,  by using the geodesic normal coordinates to $H_{3m} \subseteq M_{H_{3m}}$, we can identify $H_{3m} \times [-1, -1+ \varepsilon']$ with a neighborhood of $H_{3m}$,  in $M_{H_{3m}}$ via a diffeomorphism $\iota$.  Let $U=\iota(H_{3m} \times (-1, -1+ \varepsilon'])$, then naturally, $\iota(H_{3m} \times [-1, -1+ \varepsilon'])=\bar{U}$. Now, we require the metric on $H_{3m} \times [-1, -1+\varepsilon')$ to be the pull-back metric obtained from that of $U$ by $\iota$.
In the same way, we can construct a metric near the boundary $H_{3m}\times \{2\}$ of $H_{3m}\times [-1, 2]$, i.e.,\ $H_{3m}\times (2-\varepsilon'', 2]$.
Meanwhile, on $H_{3m}\times [0,1]$, we give the product metric {constructed} by $g^{TH_{3m}}$ and the standard metric on $[0,1]$. 
Finally, the metric on $H_{3m}\times [-1,2]$ is a smooth extension of the metrics on the above three pieces. 

Assume $f(Y_j)=p_j \in S^{2k-1}(1)$, $j= 1,\dots, l$. We choose a point $p_0\in S^{2k-1}(1)$ and for $j= 1,\dots,l$, pick a curve $\xi_j(\tau), 0\leq \tau\leq 1$, connecting $p_j$ and $p_0$ such that $\xi_i(\tau)\cap \xi_j(\tau')=\emptyset, 0<\tau, \tau'<1, i\neq j$. Following \cite{SWZ}, for $(y,\tau)\in H_{3m}\times [-1,2]$, $j=1,\dots,l$, we define\footnote{A similar trick also appears in \cite{CZ}.}
\begin{align}\label{r3.4}
f(y,\tau)=
\begin{cases}
p_j,\ (y,\tau)\in (Y_j\cap H_{3m})\times [-1,0],\\
\xi_j(\tau),\ (y,\tau)\in (Y_j\cap H_{3m})\times [0,1],\\
p_0,\ (y,\tau)\in (Y_j\cap H_{3m})\times [1,2].
\end{cases}
\end{align}
Note that some points of $\{p_j\}_{j=1}^l$ may coincide. Without loss of generality, we assume that  $(0,\cdots,0,\pm 1)\notin \xi_j(\tau), 1\leq j\leq l$.

Recall that $H_{3m}\times [-1,-1+\varepsilon')$ can be identified with a neighborhood $U$ of $H_{3m}$ in $M_{H_{3m}}$. Under such an identification, the above $f(y,\tau)$ coincides with $f$ on $U$. Thus, $f$ can be extended to a map on $M_{H_{3m}}\cup (H_{3m}\times [-1,2])$ via $f(y,\tau)$. Denote such a map on $M_{H_{3m}}\cup (H_{3m}\times [-1,2])$ by $f_l$.

Let $M'_{H_{3m}}$ be another copy of $M_{H_{3m}}$ with the same metric and the opposite orientation.
Let $\iota'$ be the diffeomorphism, the isometry actually, from $H_{3m}\times [2-\varepsilon'', 2)$ to a neighborhood of $\partial M'_{H_{3m}}$, $U'$, in $M'_{H_{3m}}$.
On the disjoint union,
\begin{equation*}
  M^{\circ}_{H_{3m}} \sqcup H_{3m} \times (-1,2) \sqcup M'^{,\circ}_{H_{3m}},
\end{equation*}
we consider the equivalence relation $\sim$ given by $x_1 \sim x_2$ if and only if $x_1\in U^{\circ}$, $x_2\in H_{3m}\times (-1, -1+\varepsilon')$ (resp.\ $x_1\in U'^{,\circ}$, $x_2\in H_{3m} \times (2-\varepsilon'', 2)$) and $x_1= \iota(x_2)$ (resp.\ $x_1 = \iota'(x_2)$).
As a set, we define the gluing manifold $\widehat{M}_{H_{3m}}$ to be
\begin{equation*}
  \widehat{M}_{H_{3m}} = (M^{\circ}_{H_{3m}} \sqcup H_{3m}\times (-1,2) \sqcup M'^{,\circ}_{H_{3m}})/\sim,
\end{equation*}
endowed with the differentiable structure associated with the open cover 
\[
\{M^{\circ}_{H_{3m}}, H_{3m}\allowbreak\times (-1,2), M'^{,\circ}_{H_{3m}}\}.
\]
Moreover, since $\iota$ and $\iota'$ are isometries with respect to the metrics on $M^{\circ}_{H_{3m}}$, $H_{3m}\times (-1,2)$, and $M'^{,\circ}_{H_{3m}}$, $\widehat{M}_{H_{3m}}$ also inherits a metric from this open cover.
From now on, we view $M_{H_{3m}}$, $M'_{H_{3m}}$ and $H_{3m}\times [-1,2]$ as submanifolds of $\widehat{M}_{H_{3m}}$. The map $f_l$ can be extended to $\widehat{M}_{H_{3m}}$ by setting $f_l(M'_{H_{3m}})={p_0}$.
We still denote the map on $\widehat{M}_{H_{3m}}$ by $f_l$. The map $f_l$ has the following properties: 
\begin{align}\label{6/3}{\rm Supp}({\rm d}f_l)\subseteq {\rm Supp}({\rm d}f)\cup (H_{3m}\times [0,1]),\ {\rm deg}(f_l)={\rm deg}(f)\neq 0.
\end{align}

For any $\beta>0$, let $g^{TM}_{\beta}$ be the Riemannian metric on $M$ defined by
\begin{align}\label{1}
g^{TM}_{\beta}=\beta^2g^{TM}.
\end{align}

Let $g^{TH_{3m}}$ be the induced metric on $H_{3m}$ by (\ref{1}) with $\beta=1$ and ${\rm d}t^2$ be the standard metric on $[0,1]$. By the construction of $\widehat{M}_{H_{3m}}$, we can define a smooth metric $g_{\beta}^{T\widehat{M}_{H_{3m}}}$ on $\widehat{M}_{H_{3m}}$ in the following way:
\begin{align}\label{r3.6}
  \left . g_{\beta}^{T\widehat{M}_{H_{3m}}}\right|_{M_{H_{3m}}} = g^{TM}_{\beta}, \ 
   \left . g_{\beta}^{T\widehat{M}_{H_{3m}}}\right|_{M'_{H_{3m}}} = g^{TM_{H_{3m}}'}, \ 
 \left .   g_{\beta}^{T\widehat{M}_{H_{3m}}}\right|_{H_{3m}\times [0,1]} = g^{TH_{3m}}\oplus {{\rm d} t^2},
   \end{align}
and then paste these metrics together. Let $\nabla_\beta^{T\widehat{M}_{H_{3m}}}$ be the Levi-Civita connection on $T\widehat{M}_{H_{3m}}$ associated to the metric $g_{\beta}^{T\widehat{M}_{H_{3m}}}$.

{Figure~\ref{fig:tp} helps to explain this gluing procedure.}

\begin{figure}[ht]
  \centering
  \includegraphics[scale=0.9]{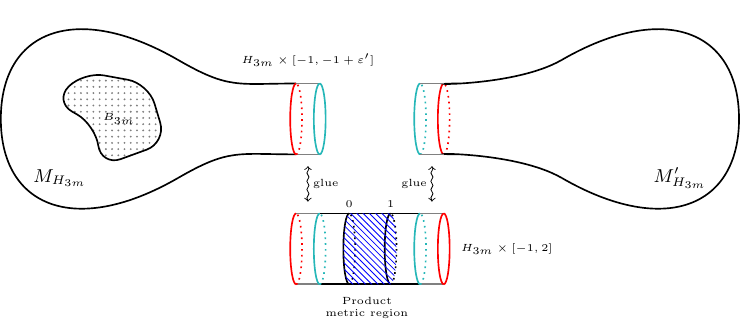}
  \caption{Gluing three parts.}
  \label{fig:tp}
\end{figure}

\subsection{A family of deformed twisted Dirac operators}

Recall that the ${\mathbb Z}_2$-graded vector bundle $E=E_+\oplus E_-$ over $S^{2k-1}(1)$ has been discussed in   Section 2. 

Let $\mE_{3m}=\mE_{3m,+}\oplus \mE_{3m,-}=f^*_l(E_+)\oplus f^*_l(E_-)$ be the pull-back bundle with the family of pull-back connections $\nabla^{\mE_{3m},u}=\nabla^{\mE_{3m,+},u}\oplus \nabla^{\mE_{3m,-},u}=f^*_l(\nabla^{E_+,u})\oplus f^*_l(\nabla^{E_-,u})$, $0\leq u\leq 1$, and the pull-back metric. Let $(S_{\beta}(T\widehat{M}_{H_{3m}}),\nabla^{S_{\beta}(T\widehat{M}_{H_{3m}})})$ be the spinor bundle associated to the metric $g_{\beta}^{T\widehat{M}_{H_{3m}}}$, where $\nabla^{S_{\beta}(T\widehat{M}_{H_{3m}})}$ is the Hermitian connection on $S_\beta(T\widehat{M}_{H_{3m}})$ induced by $\nabla_{\beta}^{T\widehat{M}_{H_{3m}}}$.

Let $\nabla^{S_{\beta}(T\widehat{M}_{H_{3m}})\otimes \mE_{3m,\pm},u}$, $0\leq u\leq 1$, be the family of connections
on $S_{\beta}(T\widehat{M}_{H_{3m}})\otimes \mE_{3m,\pm}$
 induced by $\nabla^{S_{\beta}(T\widehat{M}_{H_{3m}})}$ and $\nabla^{\mE_{3m,\pm},u}$. 
 Let 
 \[
 \nabla^{S_{\beta}(T\widehat{M}_{H_{3m}})\otimes \mE_{3m},u}=\nabla^{S_{\beta}(T\widehat{M}_{H_{3m}})\otimes \mE_{3m,+},u}\oplus \nabla^{S_{\beta}(T\widehat{M}_{H_{3m}})\otimes \mE_{3m,-},u}
 \]
  be the family of connections on 
 \begin{align}\label{k3.10}
S_{\beta}(T\widehat{M}_{H_{3m}})\otimes \mE_{3m}=S_{\beta}(T\widehat{M}_{H_{3m}})\otimes \mE_{3m,+}\oplus S_{\beta}(T\widehat{M}_{H_{3m}})\otimes \mE_{3m,-}.
\end{align}

 Let    $D^{\mE_{3m},u}_{\beta}$, $0\leq u\leq 1$, acting on $S_{\beta}(T\widehat{M}_{H_{3m}})\otimes \mE_{3m}$ be the  family of twisted Dirac operators defined by
\begin{align}\label{y3.6}
D^{\mE_{3m},u}_{\beta}=\sum_{i=1}^{2k-1}c_{\beta}(h_i)\nabla_{h_i}^{S_{\beta}(T\widehat{M}_{H_{3m}})\otimes \mE_{3m},u},
\end{align}
where $\{h_1,\cdots,h_{2k-1}\}$ is a local oriented orthonormal basis of $(T\widehat{M}_{H_{3m}},g_{\beta}^{T\widehat{M}_{H_{3m}}})$, and $c_\beta(\cdot)$ means that the Clifford action is with respect to the metric $g^{T\widehat{M}_{H_{3m}}}_\beta$. 
\footnote{As explained in \cite[(2.2)]{BC}, with respect to the splitting 
(\ref{k3.10}),
the twisted Dirac operator in (\ref{y3.6}) has the form
$$
D^{\mE_{3m},u}_{\beta}=\left(\begin{matrix}D^{\mE_{3m,+},u}_{\beta}&& 0\\0&&-D^{\mE_{3m,-},u}_{\beta}\end{matrix}\right),
$$
in terms of the  standard (ungraded) Dirac operators.}

Following \cite{Z20}, let $U_{1\over 2}\subset M$ be the subset defined by $U_{1\over 2}=\{x\in {\rm Supp}({\rm d}f): |{\rm d}f(x)|<{1\over 2}\}$. Let $V_{1\over 2}\subset M$ be the open subset defined by $V_{1\over 2}=\{x:|\wedge^2({\rm d}f(x))|>{1\over 2}\}$, where $\wedge^2({\rm d}f)$ is the induced action of ${\rm d} f$ on the exterior product $\wedge^2(TM)$. Clearly, $\overline{U_{1\over 2}}\cap \overline{V_{1\over 2}}=\emptyset$.

Let $\varphi:M_{H_{3m}}\to [0,1]$ be a smooth function such that $\varphi=1$ on $(M_{H_{3m}}\setminus{\rm Supp}({\rm d}f))\cup U_{1\over 2}$, while $\varphi=0$ on $V_{1\over 2}$. Then we can extend $\varphi$ to $\widehat{M}_{H_{3m}}$ by setting $\varphi(\widehat{M}_{H_{3m}}\setminus M_{H_{3m}})=1$. We still denote the extended function on $\widehat{M}_{H_{3m}}$ by $\varphi$.

We choose a tangent vector field $v$ on $S^{2k-1}(1)\subset {\mathbb R}^{2k}$ such that $v$ is perpendicular to $\partial_{2k}$ and $v\neq 0$ on $\xi_j(\tau), 0\leq \tau\leq 1, 1\leq j\leq l$. Such a vector field clearly exists. 
\begin{prop}\label{p2.2}
For any $x\in S^{2k-1}(1)$, one has
$$\left[g(x),{\widetilde c} (v(x))\right]=0.$$
\end{prop}
\begin{proof}
Since $v(x)$ is perpendicular to both $\partial_{2k}$ and $x$, then 
\begin{multline}
\left[g(x),{\widetilde c} (v(x))\right]=\left[{\widetilde c} (\partial_{2k}){\widetilde c} (x),{\widetilde c} (v(x))\right]={\widetilde c} (\partial_{2k}){\widetilde c} (x){\widetilde c} (v(x))-{\widetilde c} (v(x)){\widetilde c} (\partial_{2k}){\widetilde c} (x)\\
={\widetilde c} (\partial_{2k})({\widetilde c} (x){\widetilde c} (v(x))+{\widetilde c} (v(x)){\widetilde c} (x))=0.
\end{multline}
\end{proof}

Set 
$$V={\rm Id}_{S_{\beta}\left(T\widehat{M}_{H_{3m}}\right)}\otimes \sqrt{-1}f^*_{l}(\widetilde c(v)): \Gamma\left(S_{\beta}\left(T\widehat{M}_{H_{3m}}\right){\otimes}\mE_{3m}\right)\to \Gamma\left(S_{\beta}\left(T\widehat{M}_{H_{3m}}\right){\otimes}\mE_{3m}\right).$$

For any $\varepsilon>0$, let $D^{\mE_{3m},u}_{\beta,\varepsilon}:\Gamma(S_{\beta}(T\widehat{M}_{H_{3m}}){\otimes}\mE_{3m})\to \Gamma(S_{\beta}(T\widehat{M}_{H_{3m}}){\otimes}\mE_{3m})$, $0\leq u\leq 1$, be the family of deformed twisted Dirac operators defined by
\begin{align}\label{2.5}
D^{\mE_{3m},u}_{\beta,\varepsilon}=D^{\mE_{3m},u}_{\beta}+{{\varepsilon\varphi V}\over{\beta}}.
\end{align}
Clearly, $D^{\mE_{3m},u}_{\beta,\varepsilon}$ is self-adjoint.

\begin{prop}\label{p3.1}
For $0\leq u\leq 1$, one has
\begin{align}\label{b3.8}
D^{\mE_{3m},u}_{\beta,\varepsilon}=(1-u)\left(D^{\mE_{3m},0}_{\beta}+{{\varepsilon\varphi V}\over{\beta}}\right)+u (f^*_l(g))^{-1}\left(D^{\mE_{3m},0}_{\beta}+{{\varepsilon\varphi V}\over{\beta}}\right)f^*_l(g).
\end{align}
In particular,
\begin{align}\label{b3.8a}
D^{\mE_{3m},1}_{\beta,\varepsilon}= (f^*_l(g))^{-1}\left(D^{\mE_{3m},0}_{\beta}+{{\varepsilon\varphi V}\over{\beta}}\right)f^*_l(g).
\end{align}
\end{prop}

\begin{proof}
By Proposition \ref{p2.2}, we have
\begin{align}\label{y4.2}
(1-u){{\varepsilon\varphi V}\over{\beta}}+u (f^*_l(g))^{-1}{{\varepsilon\varphi V}\over{\beta}}f^*_l(g)={{\varepsilon\varphi V}\over{\beta}}.
\end{align}
By definition, we have
\begin{align}\label{y4.3}
(1-u)D^{\mE_{3m},0}_{\beta}+u (f^*_l(g))^{-1}D^{\mE_{3m},0}_{\beta}f^*_l(g)=D^{\mE_{3m},u}_{\beta}.
\end{align}

From (\ref{2.5}), (\ref{y4.2}) and (\ref{y4.3}), we get (\ref{b3.8}).
\end{proof}

\section{The invertibility of the deformed twisted Dirac operators}

We use the notations in Section 3. Let $k^{TM}$ be the scalar curvature of $g^{TM}$. Moreover, in this section, we also assume
\begin{align}\label{f5.1}
k^{TM}\geq (2k-1)(2k-2)\ {\rm on}\ {\rm Supp}({\rm d}f)
\end{align}
and
\begin{align}\label{e5.2}
\inf(k^{TM})\geq 0.
\end{align}

Under the assumptions (\ref{f5.1}) and (\ref{e5.2}), we will show that for suitable parameters $\varepsilon$, $m$ and $\beta$, the operator $D^{\mE_{3m},u}_{\beta,\varepsilon}$ is invertible for any $u\in [0,1]$.

From (\ref{2.5}), one has
\begin{multline}\label{a2.6}
\left(D^{\mE_{3m},u}_{\beta,\varepsilon}\right)^2
=\left(D^{\mE_{3m},u}_{\beta}\right)^2+{{\varepsilon}\over{\beta}}\left[D^{\mE_{3m},u}_{\beta},\varphi V\right]+{{\varepsilon^2\varphi^2 V^2}\over{\beta^2}}
\\
=\left(D^{\mE_{3m},u}_{\beta}\right)^2+{{\varepsilon c_\beta({\rm d}\varphi)V}\over{\beta}}+{{\varepsilon\varphi}\over{\beta}}\left[D^{\mE_{3m},u}_{\beta},V\right]+{{\varepsilon^2\varphi^2 V^2}\over{\beta^2}},
\end{multline}
where we identify ${\rm d}\varphi$ with the gradient of $\varphi$ respect to the metric $g_{\beta}^{T\widehat{M}_{H_{3m}}}$. 

By (\ref{2.00}) one has $[{c}_\beta(h_i), V]=0$ (cf. \cite[(2.6)]{BC}), $1\leq i\leq 2k-1$,  thus the operator $[D^{\mE_{3m},u}_{\beta},V]$ is an operator of order zero.

Following \cite[Theorem 1.17]{GL83}, let $\phi:[0, \infty) \rightarrow [0,1]$ be a smooth function such that $\phi \equiv 1$ on $[0, 1]$, $\phi \equiv 0$ on $[2, \infty)$ and $\phi' $ is very close to $-1$ on $[1, 2]$. We define a smooth function {${\psi_m}:M_{H_{3m}}\to [0,1]$ by
  \begin{equation}\label{6.6}
    \psi_m(x) = \phi\left(\frac{d(x)}{m}\right),
  \end{equation}
  where $m\in \mathbb{N}$. We extend $\psi_m$ to $(H_{3m}\times [-1,2])\cup M'_{H_{3m}}$ by setting
  \begin{equation*}
    \psi_m\big((H_{3m}\times [-1,2])\cup M'_{H_{3m}}\big)=0.
  \end{equation*}

Following~\cite[p. 115]{BL91}, let $\psi_{m,1},\, \psi_{m,2}: \widehat{M}_{H_{3m}}\rightarrow [0,1]$ be defined by
\begin{align}\label{0.20}
  \psi_{m,1} =\frac{\psi_{m}}{\big(\psi_{m}^2+(1-\psi_{m})^2\big)^{\frac{1}{2}}},\; \psi_{m,2} =\frac{1-\psi_{m}}{\big(\psi_{m}^2+(1-\psi_{m})^2\big)^{\frac{1}{2}}}.
\end{align}
Then $\psi^2_{m,1}+\psi^2_{m,2}=1$.

From (\ref{eq:fun-d}), (\ref{6.6}) and (\ref{0.20}), for $i=1,2$, we have
\begin{equation}
  \label{eq:est-psi}
  |\nabla \psi_{m,i}|(x)\leq {\frac{C}{m}} \text{ for any }x\in \widehat{M}_{H_{3m}},
\end{equation}
where $C$ is a constant independent of $m$ and $|\cdot |$ is respect to the metric $g^{TM}$.

For any $s\in \Gamma(S_\beta(T\widehat{M}_{H_{3m}}){\otimes}\mE_{3m})$, one has
\begin{align}
\left\|\left(D^{\mE_{3m},u}_{\beta}+{{\varepsilon\varphi V}\over{\beta}}\right)s\right\|_\beta^2
=\left\|\psi_{m,1}\left(D^{\mE_{3m},u}_{\beta}+{{\varepsilon\varphi V}\over{\beta}}\right)s\right\|_\beta^2+\left\|\psi_{m,2}\left(D^{\mE_{3m},u}_{\beta}+{{\varepsilon\varphi V}\over{\beta}}\right)s\right\|_\beta^2,
\end{align}
from which one gets,
\begin{multline}\label{0.22}
2\left\|\left(D^{\mE_{3m},u}_{\beta}+{{\varepsilon\varphi V}\over{\beta}}\right)s\right\|_\beta^2
 \geq \left\|\left(D^{\mE_{3m},u}_{\beta}+{{\varepsilon\varphi V}\over{\beta}}\right)\psi_{m,1} s\right\|_\beta^2 +\left\| \left(D^{\mE_{3m},u}_{\beta}+{{\varepsilon\varphi V}\over{\beta}}\right)\psi_{m,2} s\right\|_\beta^2
 \\
 - 2\left\|c_\beta({\rm d}\psi_{m,1})s\right\|^2_\beta- 2\left\|c_\beta({\rm d}\psi_{m,2})s\right\|^2_\beta,
\end{multline}
where $\|\cdot\|_\beta$ means that the norm is defined with respect to  $g_\beta^{T\widehat{M}_{H_{3m}}}$ and we identify a one-form with its gradient.

On $M_{H_{3m}}$, one has, via (\ref{a2.6}) and the Lichnerowicz formula \cite{L63}, 
\begin{multline}\label{a2.13}
\left(D^{\mE_{3m},u}_{\beta}+{{\varepsilon\varphi V}\over{\beta}}\right)^2
=-\Delta^{\mE_{3m},\beta,u}+{k^{TM}\over {4\beta^2}}+{1\over 2\beta^2}\sum_{i,\,j=1}^{2k-1}R^{\mE_{3m},u}(e_i,e_j)c_\beta(\beta^{-1}e_i)c_\beta(\beta^{-1}e_j)\\
+{{\varepsilon c_\beta({\rm d}\varphi)V}\over{\beta}}+{{\varepsilon\varphi}\over{\beta}}\left[D^{\mE_{3m},u}_{\beta},V\right]+{{\varepsilon^2\varphi^2V^2}\over{\beta^2}},
\end{multline}
where $-\Delta^{\mE_{3m},\beta,u}\geq 0$ is the corresponding Bochner Laplacian, $\{e_1,\cdots, e_{2k-1}\}$ is a local oriented orthonormal basis for $(TM,g^{TM})$ and $R^{\mE_{3m,u}}=(\nabla^{\mE_{3m,u}})^2$. Note that the curvature term $R^{\mE_{3m},u}$ in (\ref{a2.13}) is supported in $M_{H_{3m}}$.

We recall the following result from \cite[(3.6)]{LSW}.
\begin{prop}\label{p5.1}
Let $\{\epsilon_i\}_{i=1}^{2k-1}$ be a local oriented orthonormal basis of $TS^{2k-1}(1)$ and $\{\epsilon^i\}_{i=1}^{2k-1}$ be the dual basis. Let $R^{f^*E,u}=\left(f^*\left(\nabla^{E,u}\right)\right)^2$ be the curvature  of $f^*(\nabla^{E,u})$. Then
$$R^{f^*E,u}=-u(1-u)\sum_{i,\, j=1}^{2k-1}f^*(\epsilon^i\wedge \epsilon^j)f^*\left(\widetilde c(\epsilon_i)\widetilde c(\epsilon_j)\right).$$
\end{prop}

Let $|\wedge^{2}({\rm d} f)|_\beta$ be the norm of $\wedge^2({\rm d}f)$ with respect to the metric $g^{TM}_\beta$, then on $M_{H_{3m}}$, one has
\begin{align}\label{d5.11}
\left|\wedge^{2}({\rm d} f)\right|_\beta={1\over \beta^2}\left|\wedge^2({\rm d}f)\right| .
\end{align}

By Proposition \ref{p5.1}, (\ref{d5.11}) and proceeding as in \cite[(3.7)-(3.8)]{LSW}, one has on $M_{H_{3m}}$ that
\begin{multline}\label{a2.14}
{1\over 2\beta^2}\sum_{i,\,j=1}^{2k-1}R^{\mE_{3m},u}(e_i,e_j)c_\beta(\beta^{-1}e_i)c_\beta(\beta^{-1}e_j)
\geq -{(2k-1)(2k-2)}u(1-u)\left|\wedge^2({\rm d}f)\right|_\beta\\
= -{(2k-1)(2k-2)\over \beta^2}u(1-u)\left|\wedge^2({\rm d}f)\right|
 \geq -{(2k-1)(2k-2)\over 4\beta^2}\left|\wedge^2({\rm d}f)\right|.
\end{multline}

By the definition of $\varphi$, one has $\varphi=0$ on ${V_{1\over 2}}$. Then from (\ref{6.0}),  (\ref{f5.1}), (\ref{a2.13}) and (\ref{a2.14}), one has that near any $x\in {V_{1\over 2}}$, 
\begin{align}\label{a2.15}
\left(D^{\mE_{3m},u}_{\beta,\varepsilon}\right)^2+\Delta^{\mE_{3m},\beta,u}\geq 0.
\end{align}

Near any $x\in {\rm Supp}({\rm d}f)\setminus \overline{V_{1\over 2}}$, from (\ref{f5.1}), (\ref{a2.13}) and (\ref{a2.14}), one has
\begin{align}\label{a2.16}
\left(D^{\mE_{3m},u}_{\beta,\varepsilon}\right)^2+\Delta^{\mE_{3m},\beta,u}\geq {{(2k-1)(2k-2)}\over{8\beta^2}}
+{{\varepsilon c_\beta({\rm d}\varphi)V}\over{\beta}}+{{\varepsilon\varphi}\over{\beta}}\left[D^{\mE_{3m},u}_{\beta},V\right]+{{\varepsilon^2\varphi^2 V^2}\over{\beta^2}}.
\end{align}

By \cite[(2.10)]{Z20} and the definition of $\varphi$, one has near any $x\in M_{H_{3m}}\setminus {\rm Supp}({\rm d}f)$, 
\begin{align}\label{z4.14}
{1\over 2\beta^2}\sum_{i,\,j=1}^{2k-1}R^{\mE_{3m},u}(e_i,e_j)c_\beta(\beta^{-1}e_i)c_\beta(\beta^{-1}e_j)= 0
\end{align}
and
\begin{align}\label{s4.15}
{{\varepsilon c_\beta({\rm d}\varphi)V}\over{\beta}}=0,\ \ {{\varepsilon\varphi}\over{\beta}}\left[D^{\mE_{3m},u}_{\beta},V\right]=0.
\end{align}
Then by (\ref{a2.13}), (\ref{z4.14}) and (\ref{s4.15}), one has near any $x\in M_{H_{3m}}\setminus {\rm Supp}({\rm d}f)$, 
\begin{align}\label{a2.17}
\left(D^{\mE_{3m},u}_{\beta,\varepsilon}\right)^2+\Delta^{\mE_{3m},\beta,u}= {k^{TM}\over 4\beta^2}+{{\varepsilon^2 V^2}\over{\beta^2}}.
\end{align}

Recall that we extend the map $f:M_{H_{3m}}\to S^{2k-1}(1)$ to the map $f_l:\widehat{M}_{H_{3m}}\to S^{2k-1}(1)$ by (\ref{r3.4}) and $f_l(M'_{H_{3m}})=p_0$. The tangent vector field $v$ on $S^{2k-1}(1)$ is chosen such that $v$ is perpendicular to $\partial_{2k}$ and $v\neq 0$ on $\xi_j(\tau), 0\leq \tau\leq 1, 1\leq j\leq l$.
Then there exists $\delta>0$ such that on $\widehat{M}_{H_{3m}}\setminus K$, 
\begin{align}\label{e5.17}
V^2\geq \delta.
\end{align}

Since ${\rm Supp}(\psi_{m,1}s)\subseteq M_{H_{3m}}$ and we have assumed that $\inf (k^{TM})\geq 0$, from (\ref{a2.15})-(\ref{e5.17}) and the compactness of ${\rm Supp}({\rm d}f)$, we see that when $\varepsilon>0$ is small enough (independent of $m\geq 1$ and $\beta\leq 1$), there is a smooth nonnegative endomorphism $a_{\varepsilon}$ of $(S_\beta(T\widehat{M}_{H_{3m}}){\otimes}\mE_{3m})|_{M_{H_{3m}}}$ and a constant $c_\varepsilon>0$ such that
\begin{align}\label{y5.14}
a_{\varepsilon}\geq c_\varepsilon {\rm Id}\ \ {\rm on}\ \ (M_{H_{3m}}\setminus K)\cup U_{1\over 2}
\end{align}
and 
\begin{multline}\label{e5.16}
\left\langle \left(D^{\mE_{3m},u}_{\beta}+{{\varepsilon\varphi V}\over{\beta}}\right)^2\psi_{m,1}s,\psi_{m,1}s\right\rangle_{\beta}
\\
\geq \left\langle-\Delta^{\mE_{3m},\beta,u}(\psi_{m,1}s),\psi_{m,1}s\right\rangle_{\beta}+\left\langle {a_{\varepsilon}\over \beta^2}(\psi_{m,1}s),\psi_{m,1}s\right\rangle_{\beta}.
\end{multline}

By (\ref{a2.6}), one has
\begin{multline}\label{z4.18}
\left\langle \left(D^{\mE_{3m},u}_{\beta}+{{\varepsilon\varphi V}\over{\beta}}\right)^2\psi_{m,2}s,\psi_{m,2}s\right\rangle_{\beta}=\\
\left\langle \left(\left(D^{\mE_{3m},u}_{\beta}\right)^2+{{\varepsilon c_\beta({\rm d}\varphi)V}\over{\beta}}+{{\varepsilon\varphi}\over{\beta}}\left[D^{\mE_{3m},u}_{\beta},V\right]+{{\varepsilon^2\varphi^2 V^2}\over{\beta^2}}\right)\psi_{m,2}s,\psi_{m,2}s\right\rangle_\beta\\
\geq \left\langle \left({{\varepsilon c_\beta({\rm d}\varphi)V}\over{\beta}}+{{\varepsilon\varphi}\over{\beta}}\left[D^{\mE_{3m},u}_{\beta},V\right]+{{\varepsilon^2\varphi^2 V^2}\over{\beta^2}}\right)\psi_{m,2}s,\psi_{m,2}s\right\rangle_\beta.
\end{multline}
Since ${\rm Supp}(\psi_{m,2}s)\subseteq \widehat{M}_{H_{3m}}\setminus K$ and $\varphi=1$ on $\widehat{M}_{H_{3m}}\setminus K$, then by (\ref{z4.18}), one has
\begin{multline}
\left\langle \left(D^{\mE_{3m},u}_{\beta}+{{\varepsilon\varphi V}\over{\beta}}\right)^2\psi_{m,2}s,\psi_{m,2}s\right\rangle_{\beta}\\
\geq \left\langle\left( {\varepsilon\over\beta}\left[D^{\mE_{3m},u}_{\beta},V\right]+{{\varepsilon^2 V^2}\over{\beta^2}}\right)\psi_{m,2}s,\psi_{m,2}s\right\rangle_{\beta}\\
\geq \left\langle{\varepsilon\over\beta}\left[D^{\mE_{3m},u}_{\beta},V\right]\psi_{m,2}s,\psi_{m,2}s\right\rangle_{\beta}+\left\langle{{\varepsilon^2 V^2}\over{\beta^2}}\psi_{m,2}s,\psi_{m,2}s\right\rangle_{\beta}.
\end{multline}

Since ${\rm Supp}({\rm d}f_l)\cap (\widehat{M}_{H_{3m}}\setminus K)\subseteq H_{3m}\times [0,1]$, one has
\begin{align}\label{z4.20}
 \left\langle{\varepsilon\over\beta}\left[D^{\mE_{3m},u}_{\beta},V\right]\psi_{m,2}s,\psi_{m,2}s\right\rangle_{\beta}=\left\langle {\varepsilon\over\beta}\left[D^{\mE_{3m},u}_{\beta},V\right]\psi_{m,2}s,\psi_{m,2}s \right\rangle_{\beta,H_{3m}\times [0,1]}.
  \end{align}

By (\ref{r3.6}), the metric $g_{\beta}^{T\widehat{M}_{H_{3m}}}$ is independent of $\beta$ on $H_{3m}\times [0,1]$, then by (\ref{z4.20}), there exists a constant $C_m>0$ such that
\begin{align}
 \left|\left\langle{\varepsilon\over\beta}\left[D^{\mE_{3m},u}_{\beta},V\right]\psi_{m,2}s,\psi_{m,2}s\right\rangle_{\beta}\right|\leq  {C_m \varepsilon\over \beta}\left\|\psi_{m,2}s\right\|^2_{\beta }.\end{align}

By (\ref{e5.17}) and ${\rm Supp}(\psi_{m,2}s)\subseteq \widehat{M}_{H_{3m}}\setminus K$, one has
\begin{multline}\label{z4.22}
\left\langle{{\varepsilon^2 V^2}\over{\beta^2}}\psi_{m,2}s,\psi_{m,2}s\right\rangle_{\beta}=\left\langle {{\varepsilon^2 V^2}\over{\beta^2}}\psi_{m,2}s,\psi_{m,2}s\right\rangle_{\beta,\widehat{M}_{H_{3m}}\setminus K}
\geq {\varepsilon^2 \delta\over {\beta^2}}\|\psi_{m,2}s\|^2_{\beta,\widehat{M}_{H_{3m}}\setminus K}
\\={{\varepsilon^2\delta}\over{\beta^2}}\|\psi_{m,2}s\|^2_{\beta }.
\end{multline}

By (\ref{z4.18})-(\ref{z4.22}), one has
\begin{align}\label{e5.18}
\left\langle \left(D^{\mE_{3m},u}_{\beta}+{{\varepsilon\varphi V}\over{\beta}}\right)^2\psi_{m,2}s,\psi_{m,2}s\right\rangle_{\beta}
\geq {{\varepsilon^2\delta}\over{\beta^2}}\|\psi_{m,2}s\|^2_{\beta }-{C_m \varepsilon\over \beta}\left\|\psi_{m,2}s\right\|^2_{\beta }.
\end{align}

From the definitions of $\psi_{m,1}$, $\psi_{m,2}$ and 
(\ref{eq:est-psi}), there exists a constant $C'>0$ such that
\begin{align}\label{e5.19}
\left\|c_\beta({\rm d}\psi_{m,1})s\right\|^2_\beta+\left\|c_\beta({\rm d}\psi_{m,2})s\right\|^2_\beta\leq {C'\over {\beta^2 m^2}}\|s\|^2_{\beta,\widehat{M}_{H_{3m}}\setminus K }.
\end{align}

Recall that $c_\varepsilon$ defined in (\ref{y5.14}) is a positive lower bound of $a_\varepsilon$ on $(M_{H_{3m}}\setminus K)\cup U_{1\over 2}$. We choose a sufficiently large $m\in \mathbb{N}$ such that
\begin{align}
{{\varepsilon^2{\delta}}\over 4}>{{C'}\over { m^2}}\ \ {\rm and\ on}\ (M_{H_{3m}}\setminus K)\cup U_{1\over 2},\ \ 
{{c_\varepsilon}\over 4}>{C'\over { m^2}}.
\end{align}
Then we choose a sufficiently small $\beta>0$ such that
\begin{align}\label{z4.27}
{\varepsilon^2 \delta\over 2\beta}>{C_m \varepsilon }.
\end{align}

Now we fix such $m$ and $\beta$. By (\ref{0.22}), (\ref{e5.16}) and (\ref{e5.18})-(\ref{z4.27}), one has

\begin{multline}\label{x4.27}
   2\left\|\left(D^{\mE_{3m},u}_{\beta}+{{\varepsilon\varphi V}\over{\beta}}\right)s\right\|_\beta^2\\
 \geq \left\langle-\Delta^{\mE_{3m},\beta,u}(\psi_{m,1}s),\psi_{m,1}s\right\rangle_{\beta}+\left\langle {a_{\varepsilon}\over \beta^2}(\psi_{m,1}s),\psi_{m,1}s\right\rangle_{\beta}\\
 \hspace{3cm}+ {{\varepsilon^2\delta}\over{\beta^2}}\|\psi_{m,2}s\|^2_{\beta }-{C_m \varepsilon\over \beta}\left\|\psi_{m,2}s\right\|^2_{\beta }- {2C'\over {\beta^2 m^2}}\|s\|^2_{\beta,\widehat{M}_{H_{3m}}\setminus K }\\
 \geq \left\langle-\Delta^{\mE_{3m},\beta,u}(\psi_{m,1}s),\psi_{m,1}s\right\rangle_{\beta}+ \left\langle {a_{\varepsilon}\over \beta^2}(\psi_{m,1}s),\psi_{m,1}s\right\rangle_{\beta}\\
 +{{\varepsilon^2\delta}\over{2\beta^2}}\|\psi_{m,2}s\|^2_{\beta} -{2C'\over {\beta^2 m^2}}\|s\|^2_{\beta,\widehat{M}_{H_{3m}}\setminus K }\\
 \geq \left\langle-\Delta^{\mE_{3m},\beta,u}(\psi_{m,1}s),\psi_{m,1}s\right\rangle_{\beta}+ {4C'\over \beta^2 m^2}\|\psi_{m,1}s\|^2_{\beta,(M_{H_{3m}}\setminus K)\cup U_{1\over 2}} \\
 + {{\varepsilon^2\delta}\over{2\beta^2}}\|\psi_{m,2}s\|^2_{\beta} -{2C'\over {\beta^2 m^2}}\|s\|^2_{\beta,\widehat{M}_{H_{3m}}\setminus K }\\   
 \geq  \left\langle-\Delta^{\mE_{3m},\beta,u}(\psi_{m,1}s),\psi_{m,1}s\right\rangle_{\beta}+ {2C'\over {\beta^2 m^2}}\|\psi_{m,1}s\|^2_{\beta,(M_{H_{3m}}\setminus K)\cup U_{1\over 2}}\\
 +\left ({{\varepsilon^2\delta}\over{2\beta^2}} -{2C'\over {\beta^2 m^2}}\right)\|\psi_{m,2}s\|^2_{\beta},
   \end{multline}
where in the last inequality we used
\begin{align}\nonumber
\|\psi_{m,1}s\|^2_{\beta,M_{H_{3m}}\setminus K}+\|\psi_{m,2}s\|^2_{\beta}=\|\psi_{m,1}s\|^2_{\beta,\widehat{M}_{H_{3m}}\setminus K}+\|\psi_{m,2}s\|^2_{\beta,\widehat{M}_{H_{3m}}\setminus K} = \|s\|^2_{\beta,\widehat{M}_{H_{3m}}\setminus K}.
\end{align}

Let $s\in {\rm ker}(D^{\mE_{3m},u}_{\beta,\varepsilon})$ for some $u\in [0,1]$, then by (\ref{x4.27}), one has that
$\psi_{m,2}s=0$,
$\psi_{m,1}s$ is parallel on $M_{H_{3m}}$ and hence has constant length on $M_{H_{3m}}$ and $ \psi_{m,1}s=0$ on $(M_{H_{3m}}\setminus K)\cup  U_{1\over 2}$. Since ${\rm Supp}(\psi_{m,1}s)\subseteq M_{H_{3m}}$, then one has $\psi_{m,1}s=0$. Then $\|s\|_\beta^2=\|\psi_{m,1}s\|_\beta^2+\|\psi_{m,2}s\|_\beta^2=0$. So there exist $\varepsilon>0$, $m\in \mathbb{N}$ and $\beta>0$ such that the operator $D^{\mE_{3m},u}_{\beta,\varepsilon}$ is invertible for any $u\in [0,1]$, under the condition (\ref{e5.2}).

\section{A proof of theorem \ref{th1.1}}

In this section we will give a proof of Theorem \ref{th1.1}. We will argue by contradiction.

Assume (\ref{e5.2}) holds, that is, 
\begin{align}\label{h5.1}
\inf(k^{TM})\geq 0.
\end{align}

Let $\eta(D^{\mE_{3m},u}_{\beta,\varepsilon})$ be the $\eta$-invariant of $D^{\mE_{3m},u}_{\beta,\varepsilon}$ in the sense of \cite{APS}. We choose  suitable $\varepsilon$, $m$ and $\beta$ as in Section 4 so that $D^{\mE_{3m},u}_{\beta,\varepsilon}$ is invertible for $0\leq u\leq 1$. Then by \cite[Lemma 3.4]{KL}, $\eta(D^{\mE_{3m},u}_{\beta,\varepsilon})$ is a smooth function of $0\leq u\leq 1$. Clearly,  for $0\leq u\leq 1$, 
\begin{align}\label{d4.6}
{\eta}\left(D^{\mE_{3m},u}_{\beta,\varepsilon}\right)-{\eta}\left(D^{\mE_{3m},0}_{\beta,\varepsilon}\right)
=\int_{0}^{u}{{\rm d}\over{{\rm d}s}}{\eta}\left(D^{\mE_{3m},s}_{\beta,\varepsilon}\right){\rm d}s.
\end{align}

By \cite[(2.29) and (2.30)]{BF}, we have the asymptotic expansion
\begin{multline}
{\rm Tr}\left[\left({{\rm d}\over {{\rm d}}s}D^{\mE_{3m},s}_{\beta,\varepsilon}\right)\exp\left(-t\left(D^{\mE_{3m},s}_{\beta,\varepsilon}\right)^2\right)\right]\\
={{a_{-(2k-1)/2}(s)}\over {t^{(2k-1)/2}}}+\cdots +{{a_{-1/2}(s)}\over{t^{1/2}}}+O(t^{1/2},s),\ \ t\to 0^{+}
\end{multline}
and 
\begin{align}\label{d4.8}
{{\rm d}\over{{\rm d}s}}{\eta}\left(D^{\mE_{3m},s}_{\beta,\varepsilon}\right)={{-2a_{-1/2}(s)}\over{\sqrt{\pi}}}.
\end{align}

 Let 
$R_\beta^{T\widehat{M}_{H_{3m}}}=(\nabla_\beta^{T\widehat{M}_{H_{3m}}})^2$
be the curvature of the Levi-Civita connection $\nabla_\beta^{T\widehat{M}_{H_{3m}}}$. Recall that the $\widehat{A}$-form is defined by (cf. \cite[(1.19)]{Z})
$$\widehat{A}\left(T\widehat{M}_{H_{3m}},\nabla_\beta^{T\widehat{M}_{H_{3m}}}\right)={\rm det}^{1/2}\left({{R_\beta^{T\widehat{M}_{H_{3m}}}/{4\pi \sqrt{-1}}}\over{\sinh\left({R_\beta^{T\widehat{M}_{H_{3m}}}/{4\pi \sqrt{-1}}}\right)}}\right).$$

Recall that $\{h_1,\cdots,h_{2k-1}\}$ is a local oriented orthonormal basis of $(T\widehat{M}_{H_{3m}},g_{\beta}^{T\widehat{M}_{H_{3m}}})$. Then by Lemma \ref{le2.1}, (\ref{a2.6}) and $[c_\beta(h_i),V]=0$, proceeding as  in \cite[Theorem 2.7 and Lemma 2.11]{BC}, \cite[Theorem 3.2]{DZ} and \cite[p. 499-500]{Ge}, we have
\begin{multline}\label{d4.9}
{{a_{-1/2}(s)}\over{\sqrt{\pi}}}
=
\int_{\widehat{M}_{H_{3m}}}\widehat{A}\left(T\widehat{M}_{H_{3m}},\nabla_{\beta}^{T\widehat{M}_{H_{3m}}}\right){\rm tr_s}\left[{f^*_l(\omega)\over {2\pi \sqrt{-1}}}\exp\left({s(1-s)f^*_l(\omega)^2}\over{2\pi \sqrt{-1}}\right)\right].
\end{multline}

From (\ref{d4.9}), since $H^i(S^{2k-1}(1))=0$, $i\neq 0, 2k-1$, by proceeding as in \cite[pp. 349]{GL83}, we have
\begin{multline}\label{d4.10}
\int_{0}^{u}{{a_{-1/2}(s)}\over{\sqrt{\pi}}}{\rm d}s
=\int_{0}^u\int_{\widehat{M}_{H_{3m}}}\left({\rm tr_s}\left[{f^*_l(\omega)\over {2\pi \sqrt{-1}}}\exp\left({s(1-s)f^*_l(\omega)^2}\over{2\pi \sqrt{-1}}\right)\right]\right)_{(2k-1)}{\rm d}s\\
=\int_{\widehat{M}_{H_{3m}}}{1\over {(k-1)!}}\left(1\over{2\pi \sqrt{-1}}\right)^{k}f^*_l\left({\rm tr_s}(\omega^{2k-1})\right)\int_{0}^{u}s^{k-1}(1-s)^{k-1} {\rm d}s.
\end{multline}

By (\ref{d4.6}), (\ref{d4.8}) and (\ref{d4.10}), we get, 
\begin{multline}\label{f5.11}
{\eta}\left(D^{\mE_{3m},1}_{\beta,\varepsilon}\right)-{\eta}\left(D^{\mE_{3m},0}_{\beta,\varepsilon}\right)\\
=-2\int_{\widehat{M}_{H_{3m}}}{1\over {(k-1)!}}\left(1\over{2\pi \sqrt{-1}}\right)^{k}f^*_l\left({\rm tr_s}(\omega^{2k-1})\right)\int_{0}^{1}s^{k-1}(1-s)^{k-1} {\rm d}s\\
=-2\int_{\widehat{M}_{H_{3m}}}{1\over {(k-1)!}}\left(1\over{2\pi \sqrt{-1}}\right)^{k}f^*_l\left({\rm tr_s}(\omega^{2k-1})\right){{\left((k-1)!\right)^2}\over{(2k-1)!}}\\
=-2\left({1\over{2\pi\sqrt{-1}}}\right)^k{{(k-1)!}\over{(2k-1)!}}{\rm deg}(f_l)\int_{S^{2k-1}(1)}{\rm tr_s}(\omega^{2k-1}).
\end{multline}

From (\ref{d4.11}), (\ref{b3.8a}),  (\ref{f5.11}) and the fact that ${\rm deg}(f_l)={\rm deg}(f)$, we get
$$
  {\rm deg}(f)=0,
$$
which contradicts the assumption that ${\rm deg}(f)\neq 0$. 
Thus (\ref{h5.1}) does not hold. The proof of Theorem \ref{th1.1} is complete.

\begin{rem}
One may also use the $\eta$-invariant on the underlying noncompact manifold directly to deal with the problem considered in this paper. 
\end{rem}

$\ $

\noindent{\bf Acknowledgments.} We are very grateful to the referees for their very careful reading and very valuable suggestions for this paper. This paper is partially supported by the National Key R\&D Program of China (No. 2024YFA1013202), NSFC Grant No. 12401062, 12425106, 12471049, 12101361, 12621005 and 11931007, Fundamental Research Funds for the Central University, Nankai University Grant No. 63241429, the Project of Young Scholars of Shandong University and Nankai Zhide Foundation.

\end{document}